\documentclass[10pt]{article}
\usepackage[usenames]{color}
\usepackage{amssymb}
\usepackage{amsfonts,amsmath}
\usepackage{multicol}
\usepackage{graphicx}
\usepackage{amsthm}
\usepackage{verbatim}
\usepackage{float}
\usepackage{forest,adjustbox}
\usepackage{color}
\usepackage{mathrsfs}
\usepackage{amsmath,url}
\usetikzlibrary{arrows,decorations.pathmorphing,decorations.markings,backgrounds,positioning,fit,petri,calc}
\usepackage{tikz}
\usetikzlibrary{arrows.meta}
\usepackage{amsmath, epsfig}

\theoremstyle{plain}
\newtheorem{thm}{Theorem}
\newtheorem{lemma}[thm]{Lemma}

\newtheorem{definition}[thm]{Definition}
\newtheorem{cor}[thm]{Corollary}
\newtheorem{prop}[thm]{Proposition}

\newtheorem{remark}[thm]{Remark}

\newtheorem{conjecture}[thm]{Conjecture}
\newtheorem{claim}{Claim}
\newtheorem{observation}{Observation}
\newtheorem*{ques*}{Question}
\usetikzlibrary{graphs}
\newcommand{\AZ}[1]{{#1}}
\title{On Seymour's and Sullivan's Second Neighbourhood Conjectures}

\author{
	Jiangdong Ai\thanks{School of Mathematical Sciences and LPMC, Nankai University. {\tt jd@nankai.edu.cn}.} \and Stefanie Gerke\thanks{Department of Mathematics. Royal Holloway University of London.   {\tt stefanie.gerke@rhul.ac.uk}.} \and Gregory Gutin \thanks{Department of Computer Science. Royal Holloway University of London. {\tt g.gutin@rhul.ac.uk}.} \and Shujing Wang \thanks {School of Mathematics and Statistics, Central China Normal University. {\tt wang06021@126.com}.} \and Anders Yeo \thanks {Department of Mathematics and Computer Science, University of Southern Denmark. {\tt andersyeo@gmail.com}.} \and Yacong Zhou\thanks{Department of Computer Science. Royal Holloway University of London. {\tt Yacong.Zhou.2021@live.rhul.ac.uk}.} }
\date{}

\begin{document}
	
	\maketitle
	
	\begin{abstract}
For a vertex $x$ of a digraph, $d^+(x)$ ($d^-(x)$, resp.) is the number of vertices at distance 1 from (to, resp.) $x$ and $d^{++}(x)$ is the number of vertices at distance 2 from $x$. 
In 1995, Seymour conjectured that for any oriented graph $D$ there exists a vertex $x$ such that $d^+(x)\leq d^{++}(x)$. In 2006, Sullivan conjectured that there exists a vertex $x$ in $D$
such that $d^-(x)\leq d^{++}(x)$. We give a sufficient condition in terms of the number of transitive triangles for an oriented graph to satisfy Sullivan's conjecture.  In particular, this implies that Sullivan's conjecture
holds for all orientations of planar graphs and of triangle-free graphs. An oriented graph $D$ is an oriented split  graph if the vertices of $D$ can be partitioned into vertex sets $X$ and $Y$ such that $X$ is an independent set and $Y$ induces a tournament.
We also show that the two conjectures hold for some families of oriented split graphs, in particular, when $Y$ induces a regular or an almost regular tournament.
		\end{abstract}
		
		\AZ{
}
		\section{Introduction}
	A directed graph $D=(V,A)$ is an {\em oriented graph} if $xy\in A$ implies that $yx\not\in A$; we also say that $D$ is an {\em orientation}
	of the {\em underlying} graph $G=(V,E)$, where $xy\in E$ if either $xy\in A$ or $yx\in A.$ 
	
	For a vertex $u$ of a directed graph $D=(V,A)$, let $N_D^-(u)=\{v\in V: vu\in A\}$ and $N_D^+(u)=\{v\in V: uv\in A\}.$
	Also, $N_D^{++}(u)=\{v\in V: uw,wv\in A \mbox{ for some } w\in A\}\setminus N_D^+(u).$ In what follows 
	we will often omit the subscript $D$ in $N_D^-(u)$, $N_D^+(u)$, etc. when $D$ is known from the context. 
	
	In 1978, Caccetta and H\"aggkvist \cite{CH} introduced the following famous conjecture.
		\begin{conjecture}\label{CHC}
			For any integer $r>0$, every digraph with $n$ vertices and the minimum out-degree at least $\frac{n}{r}$ has a cycle with length at most $r$.
		\end{conjecture}
		It is trivial when $r= 2$. For $r\geq 3$, this conjecture remains open and for $r= 3 $ it is one of the most well-known problems in graph theory. Seymour (see, e.g., \cite{DL}) proposed the following 
		conjecture which would imply the special case of Conjecture \ref{CHC} when the minimum in- and out-degree are both at least $n/3$.
	\begin{conjecture}\label{SC}
		Every oriented graph has a vertex $u$ (called a {\em Seymour vertex}, in what follows) such that
		$|N^{++}(u)| \geq |N^+(u)|$. 
	\end{conjecture} 
	Seymour's conjecture turned out to be very difficult and was confirmed only for tournaments and other
	restricted classes of digraphs, see, e.g., \cite{Fisher, FY, G1, HT}. Note that Seymour's conjecture cannot be extended to all directed graphs
	as every complete directed graph with at least two vertices (a digraph obtained from a complete graph by replacing every 
	edge $xy$ by arcs $xy,yx$) does not have a Seymour vertex. 
	
	Sullivan \cite{Sullivan} proposed the following variation of  Seymour's
	conjecture. 
	\begin{conjecture}\label{SLC}
		Every oriented graph has a vertex $u$ (called a {\em Sullivan vertex}, in what follows) such that
		$|N^{++}(u)| \geq |N^-(u)|$. 
	\end{conjecture}
	Note that this conjecture also implies the same special case of Conjecture \ref{CHC} just as Seymour's conjecture does. 
	Also note that the two conjectures coincide for Euler oriented graphs. Thus, it is possible that in general the two conjectures are of somewhat
	``equal difficulty''.

It seems that \AZ{there are only a few results on} Sullivan's conjecture (\AZ{Sullivan's conjecture was confirmed for local tournaments, bipartite tournaments, and quasi-transitive digraphs in \cite{LL, LS1,LS2}}). In this paper, we present several results on Sullivan's conjecture proving the conjecture for 
	tournaments, planar oriented graphs, some families of oriented split graphs,  and almost all oriented graphs.
	We also prove that Seymour's conjecture holds for some families of oriented split graphs (defined in the next section).
	
	Let us conclude this section with two simple results on Sullivan's conjecture and a brief discussion on the content of the other sections.
	
	A {\em tournament} is an orientation of a complete graph. 
	We say that a vertex $u$ of a digraph $D=(V,A)$ is a {\em 2-king} if $V=\{x\}\cup N^+(u)\cup N^{++}(u).$ It is well-known and it is easy to prove that every tournament
	has a 2-king, see, e.g.,  \cite{Moon,BJH} for a proof. Note that $N^-(v)\subseteq N^{++}(v)$
	for every 2-king $v$ of $V$ and hence the following holds.
	\begin{prop}\label{2king}
		Every 2-king $v$ in an oriented graph is a Sullivan vertex. In particular, every tournament has a Sullivan vertex. 
	\end{prop}
	As one can see the proof of Sullivan's conjecture for tournaments is much easier than Seymour's conjecture for tournaments \cite{Fisher,HT}. 
	
	For a real $p$ with $0<p<1$, let $\mathscr{D}(n,p)$ denote random oriented graphs with $n$ vertices in which the probability of having an arc between a pair of vertices equals $p.$
	Let $Q$ be a property of oriented graphs and let $\mathscr{D}_Q(n,p)$ denote random oriented graphs in $\mathscr{D}(n,p)$ which satisfy property $Q$.
	We say that {\em almost all oriented graphs have property $Q$} if \AZ{$\lim_{n\rightarrow \infty} |\mathscr{D}_Q(n,p)|/|\mathscr{D}(n,p)|=1$} for each $0<p<1$.
	(Our definition of almost all oriented graphs having property $Q$ is a slight extension of the usual definition where only $p=1/2$ is considered.)
	By Proposition \ref{2king}, to show that almost all oriented graphs have a Sullivan vertex,  it suffices to prove that almost all oriented graphs have a 2-king.
	\begin{prop}\label{almost}
		Almost all oriented graphs have a 2-king.
	\end{prop}
	The proof is quite simple and it is placed in Appendix.
	
	The rest of the paper is organised as follows: 
	In the next section, we will introduce additional terminology and notation. In Section \ref{TP}, we will prove that all planar oriented graphs satisfy Conjecture \ref{SC} and Conjecture \ref{SLC} by counting the number of transitive triangles. In Section \ref{SOG}, we will prove that some families of oriented split graphs satisfy Conjecture \ref{SC} and Conjecture \ref{SLC}. Finally, in Section \ref{sec:disc} we discuss open problems.
	\section{Additional Terminology and Notation} 

	Let $D=(V,A)$ be a digraph and let $X\subseteq V.$ The subgraph of $D$ induced by $X$ is denoted by $D[X]$. 
	$V(D)=V$ and $A(D)=A.$ A vertex $u\in V(D)$ is a {\em source} if $N^-(u)=\emptyset$. Note that every source is a Sullivan vertex.
	The vertices in $N^+(x)$ ($N^-(x)$, respectively) are out-neighbours (in-neighbours, respectively) of $x.$ 
	Similarly to $N^{++}(u)$ we define
	$N^{--}(u)=\{v\in V:\ vw,wu\in A \mbox{ for some } w\in A\}\setminus N^-(u).$

	To simplify some notation we use the following: $d^+_X(u)=|N^+(u)\cap X|$, $d^-_X(u)=|N^-(u)\cap X|$, $d^{++}_X(u)=|N^{++}(u)\cap X|,$ $d^{--}_X(u)=|N^{--}(u)\cap X|,$
	$d^+(u)=d^+_V(u),$ $d^-(u)=d^-_V(u)$, $d^{++}(u)=d^{++}_V(u),$ $d^{--}(u)=d^{--}_V(u).$

	For $X,Y \subset V(D)$, $X\to Y$ means that for all $x\in X$ and $y\in Y$, $xy\in A(D)$.  When $X$ and/or $Y$ are singletons,
	we do not use brackets, e.g., $x\to y$ if $X=\{x\}$ and $Y=\{y\}.$  If $x\to y$, then $x$ {\em dominates} $y$ and $y$ {\em is dominated by} $x.$ 
	Thus, in particular, a vertex $x$ dominates all its out-neighbours. 
	
	An oriented graph $D$ is an {\em oriented split graph} if $V(D)$ can be partitioned into two sets $X$ and $Y$ such that $X$ is an independent set and $Y$ induces a tournament. 
	We will denote an oriented split graph $D$ by $D=(X, Y,A)$, where the order in the union matters. If for every $x\in X$ and $y\in Y$ either $x\to y$ or $y\to x$ then $D=(X, Y,A)$
	is a {\em complete oriented split graph}. Note that a complete oriented split graph is a multipartite tournament in which all but one partite set are singletons. 
	
	A tournament $T$ is {\em regular} if $d^+(x)=d^-(x)$ for each vertex $x$ in $T$. A tournament $T$ is {\em almost regular} if $|d^+(x)-d^-(x)|=1$ for each vertex $x$ in $T$.
	
	\section{Transitive Triangles and Planar Oriented Graphs}\label{TP}
	
	An orientation of a $K_3$ is a {\em transitive triangle} if it has a source.
	We denote the number of transitive triangles in a digraph $D=(V,A)$ by ${\rm tt}(D)$. Observe that ${\rm tt}(D)\leq |A|(n-2)/3$.
	The number of transitive triangles with source $u$ in a digraph $D$ is denoted by ${\rm tt}_u(D).$
	
	\begin{thm}\label{thm:tt}
		Let $D=(V,A)$ be an oriented graph. If ${\rm tt}(D)<|A|,$ 
		then $D$ has a Sullivan vertex. 
	\end{thm}
	\begin{proof} Suppose that $D=(V,A)$ has no Sullivan vertex. 
		Observe that for any $u\in V$, we have
		\[\sum_{v\in N^+(u)}d^+(v)={\rm tt}_u + w_u,\]
		where ${\rm tt}_u$ is the number of transitive triangles with $u$ as a source,
		$w_u$ is the number of arcs from $N^+(u)$ to $N^{++}(u)$.
		Summing up this equation over all vertices, we have
		$${\rm tt}(D) +\sum_{u\in V} w_u=\sum_{u\in V} \sum_{v\in N^+(u)}d^+(v)=\sum_{u\in V} d^-(u)d^+(u).$$
		Since $D$ has no Sullivan vertex, $d^{++}(u)\leq d^-(u)-1$ for all $u\in V$. Thus, for all $u\in V$, $w_u\leq d^+(u)d^{++}(u)\leq d^+(u)(d^-(u)-1)$. Therefore, we have
		\[\sum_{u\in V} d^+(u)=\sum_{u\in V}d^+(u)(d^-(u) - d^-(u)+1)\leq \sum_{u\in V} d^-(u)d^+(u) - \sum_{u\in V} w_u={\rm tt}(D).\]
		The above and $\sum_{u\in V} d^+(u)=|A|$ imply $|A|\le {\rm tt}(D).$
		Thus, if ${\rm tt}(D)<|A|,$ then $D$ has a Sullivan vertex. 
	\end{proof}
	Thus, every digraph with the number of transitive triangles less than the number of arcs has a Sullivan vertex. In particular, we have the following corollaries, of which the first two are immediate.
	\begin{cor}
		If $D$ does not contain any transitive triangle, then $D$ has a Sullivan vertex.
	\end{cor}
	\begin{cor}
		If $D$ is an orientation of a triangle-free graph, then $D$ has a Sullivan vertex.
	\end{cor}

To prove the next corollary, we will use the following:	
	\begin{lemma}\label{lem:t}
Let $G$ be a connected planar graph with $n\geq 3$ vertices, $m$ edges and $t$ triangles. Then $t\le m-2.$	
	\end{lemma}
\begin{proof}
We prove by induction on $n$. It is easy to see that if $n=3$, the inequality holds. Assume that $n\ge 4$ and fix a plane embedding $H$ of $G.$ Let $H$ have $f$ faces. We call a triangle $C$ {\em separating} if there are vertices of $H$ on both sides of the Jordan curve determined by $C.$ First, consider the case when $H$ has no separating triangle. Then 
every triangle has an empty side and therefore bounds a face. Distinct triangles bound distinct faces, so $t\le f.$ Hence,  Euler’s formula gives $t\leq f=m-n+2\le m-2$. 

Now consider the case when $H$ has a separating triangle $C$. Let $G_1$ and $G_2$ be the plane subgraphs consisting of $C$ together with everything on the two respective sides of $C$.  Let $m_1$ and $m_2$ be the number of edges in $G_1$ and $G_2$, respectively, and $t_1$ and $t_2$ the number of triangles in $G_1$ and $G_2$, respectively. Therefore, $m=m_1+m_2-3$ and  $t=t_1+t_2-1$. In addition, as $G$ is connected, both $G_1$ and $G_2$ are connected subgraphs with more than three vertices. Thus, by induction hypothesis, 
$t_i\le m_i-2$ for $i=1,2$ and therefore $t=t_1+t_2-1\le m_1+m_2-5=m-2$. 
\end{proof}
	
	\begin{cor}\label{cor:t}
		Every planar oriented graph $D=(V,A)$ has a Sullivan vertex.
	\end{cor}
	\begin{proof} 	
Without loss of generality, we may assume that $D$  is connected.  The conclusion clearly holds when $|V|\leq 2$. Thus, we may further assume that $|V|\ge 3$. Suppose that the underlying undirected graph of $D$ has $t$ triangles. Then, by Lemma \ref{lem:t}, ${\rm tt}(D)\le t\le |A|-2$. Hence, by Theorem  \ref{thm:tt}, we are done.
\end{proof}	
	
	Note that Seymour's conjecture has been verified for digraphs with the minimum out-degree at most 6 \cite{KL}. This implies Seymour's conjecture holds for planar oriented graphs. In fact, a planar graph with $n$ vertices has at most $3n-6$ edges. Thus, the minimum out-degree of planar  oriented graphs is at most $(3n-6)/n<3$.

	\section{Oriented Split Graphs}\label{SOG}
	
	We start from the following simple but useful lemma, whose simple proof is omitted.
	\begin{lemma}\label{sums}
	For every oriented split graph $D=(X, Y,A),$ we have the following:
	$\sum_{x \in X} d^{--}_Y(x)  =  \sum_{y \in Y} d^{++}_X(y), 
			\sum_{x \in X} d^{-}_Y(x)  =  \sum_{y \in Y} d^{+}_X(y).$
	\end{lemma}

We will consider three classes of oriented split graphs in the corresponding subsections. 
	
	\subsection{Complete oriented split  graphs }
	Note that a vertex of maximum out-degree in a tournament $T$ is a 2-king of $T$. It has been shown in \cite{FY} that Seymour's second neighbourhood conjecture holds for oriented split graphs with only one vertex in the independent set and for complete oriented split graphs. We show that Sullivan's conjecture also holds for these oriented graphs.
	\begin{lemma}\label{2kingdominates}
		Let $T$ be a tournament. If $x\in V(T)$ is not a 2-king, then $N^+[x]=N^+(x)\cup \{x\}$, is dominated by a 2-king.
	\end{lemma}
	\begin{proof}
		Since $x$ is not a 2-king, $U=V(T)\setminus (N^+[x]\cup N^{++}(x))$ is not empty. Let $y$ be a 2-king in $T[U]$. Then, we can observe that $y$ dominates all the vertices in $N^+[x]$ and therefore $y$ is a 2-king of $T$.
	\end{proof}
	Note that we may always assume that an oriented graph $D$ is source-free since every source in $D$ is a Sullivan vertex. 
	\begin{thm}
		If $D=(\{x\}, Y,A)$ is a source-free oriented split graph, then $D$ has a 2-king in $Y$. 
	\end{thm}
	\begin{proof}
		Recall that $D[Y]$ is a tournament. Since $D$ is source-free, $x$ has an in-neighbour $y$ in $Y$. If $y$ is a 2-king of $D[Y]$, then $y$ is also a 2-king of $D$ and we are done. Otherwise, by Lemma \ref{2kingdominates}, $y$ is dominated by a 2-king $y'$ of $D[Y]$. Now, $y'$ is also a 2-king of $D$ because $y'\to y\to x$.
	\end{proof}
	\begin{thm}
		If $D=(X, Y,A)$ is a complete oriented split graph, then $D$ has a Sullivan vertex.
	\end{thm}
	\begin{proof}
		Let $v$ be a vertex with maximum out-degree in $D[Y]$. Since $v$ is a 2-king in $D[Y]$, $N_Y^-(v)\subseteq N_Y^{++}(v)$ and therefore $d^-_Y(v)\leq d^{++}_Y(v)$. If $N^-_X(v)\setminus N^{++}_X(v)$ is empty then we have done since $v$ is the required vertex. Otherwise, let $u\in N^-_X(v)\setminus N^{++}_X(v)$. Since $u\notin N^{++}_X(v)$, $u$ dominates all the vertices in $N_Y^+(v)$ which implies that the in-neighbours of $u$ in $Y$ are also in its second out-neighbourhood. Thus, $d_Y^-(u)\leq d_Y^{++}(u)$ and $u$ is the required vertex since all in-neighbours of $u$ are contained in $Y$.
	\end{proof}
	
	\subsection{Oriented Split Graphs with a Regular Tournament}
	Recall that a vertex of the maximum out-degree in a tournament is a 2-king. As a result, all vertices of a regular tournament are 2-kings.
	Note that (C1) and (C2) in the following theorem imply that Seymour's and Sullivan's conjectures hold for oriented split graphs with a regular tournament.
	\begin{thm}\label{thm:regular}
		Let $D=(X, Y,A(D))$ be an oriented split graph, where $Y$ induces a regular tournament $T$ in $D$. \AZ{For any $x\in V(D)$, let $d^+(x)=d^+_D(x)$ and $d^-(x)=d^-_D(x)$.} Then 
		\begin{description}
			\item[(A)] For every $x \in X$ the following two statements hold.
			\begin{description}
				\item[(A1)] $d^{++}_Y(x) \geq d^{+}(x)$ or $d^{++}_Y(x) \geq d^{-}(x)$ (or both).
				\item[(A2)] $d^{--}_Y(x) \geq d^{-}(x)$ or $d^{--}_Y(x) \geq d^{+}(x)$ (or both).
			\end{description}
			\item[(B1)] Either there exists an $x' \in X$ such that $d^{++}_Y(x') \geq d^{+}(x')$ or $d^{--}_Y(x) \geq d^{-}(x)$ for all $x \in X$.
			\item[(B2)] Either there exists an $x' \in X$ such that $d^{++}_Y(x') \geq d^{-}(x')$ or  $d^{--}_Y(x) \geq d^{+}(x)$ for all $x \in X$.
			\item[(C1)] There exists a $v \in V(D)$ such that $d^{++}(v) \geq d^{+}(v)$.
			\item[(C2)] There exists a $v \in V(D)$ such that $d^{++}(v) \geq d^{-}(v)$.
		\end{description}
	\end{thm}
	\begin{proof}
		We first prove part (A). Let $x \in X$ be arbitrary and let $A = N^+(
		x)$, $B=N^{++}(x) \cap Y$ and $C=Y \setminus (A \cup B)$.
		First assume that $C \not= \emptyset$. Note that all vertices in $C$ per definition dominate all vertices in $A$. 
		As $T$ is eulerian there are equally many arcs entering $C$ as leaving $C$ in $T$, which implies that we must have $|B| \geq |A|$ 
		(as otherwise more arcs would leave $C$ than enter $C$ in $T$), or equivalently $d^{++}_Y(x) \geq d^{+}(x)$.
		Alternatively, if $C = \emptyset$ then $N^{-}(x) \subseteq B = N^{++}(x) \cap Y$, 
		which implies that $d^{++}_Y(x) \geq d^{-}(x)$. This implies part (A1), as either
		$C = \emptyset$ or $C \not= \emptyset$. Part (A2) can be proved analogously.
		
		We now prove part (B1). 
		If for any $x\in X$, $d^{++}_Y(x)<d^{+}(x)$, then by part~(A1), we have
		$$
		d^{-}(x)\leq d^{++}_Y(x)<d^{+}(x).
		$$
		By part~(A2), this implies
		$$
		d^{--}_Y(x)\geq \min\{d^{+}(x), d^{-}(x)\}=d^{-}(x)
		$$
		as desired. Part (B2) can be proved analogously.
		

		We now prove part~(C1). By part~(B1) we only need to consider the case when $d^{--}_Y(x) \geq d^{-}(x)$ for all $x \in X$.
		Note that every vertex $y \in Y$ satisfies $d^{++}_Y(y) = d^{-}_Y(y)=d^+_Y(y)$ as $T$ is a regular tournament and therefore every vertex is a $2$-king in $T$.
		The following now holds due to Lemma \ref{sums}.
		\[
		\begin{array}{rcl}
			\sum_{y \in Y} \left( d^{++}(y)-d^{+}(y) \right) 
			& = & \sum_{y \in Y} \left( d^{++}_Y(y)-d^{+}_Y(y) \right) \\ \vspace{0.2cm}
			&   & + \sum_{y \in Y} \left( d^{++}_X(y)-d^{+}_X(y) \right) \\ \vspace{0.2cm}
			& = &  0 +  \sum_{x \in X} \left( d^{--}_Y(x)-d^{-}(x) \right) \\ 
			& \geq &  0.\\ 
		\end{array}
		\]
		This implies that for some $y \in Y$ we must have $d^{++}(y) \geq d^{+}(y)$. Part (C2) can be proved by using similar arguments with (B2).
	\end{proof}
	
	\subsection{Oriented Split  Graphs with an Almost Regular Tournament}	
	
	Let $T$ be an almost regular tournament. 	By definition, $V(T)$ can be partitioned into two sets $V^+_T$ and $V^-_T$ such that for every $u\in V^+_T$,
	$d^+(u)=d^-(u)+1$ and  for every $w\in V^-_T$, $d^-(w)=d^+(w)+1.$
\begin{prop}\label{prop:1}
		Let $T$ be an almost regular tournament such that $d^+(u)=d$ for every $u\in V^+_T$. Then $|V^+_T|=|V^-_T|=d$ and hence $T$ has $2d$ vertices.
		Moreover, for any $v\in V(T)$, we have that 
		\[
		d^{++}(v)-d^{-}(v)=\left \{\begin{array}{cc}
			0  &  \hbox{if $v$ is a 2-king;}\\
			-1  & \hbox{otherwise.}
		\end{array}
		\right.
		\]
	\end{prop}
	\begin{proof} The first part of the proposition follows from the fact that in every digraph $H$, the sum of out-degrees equals the sum of in-degrees equals the number of arcs in $H.$

	Now we prove the formula for $d^{++}(v)-d^{-}(v).$
	If $v$ is a 2-king, then since $T$ is a tournament, $N^{-}(v)= N^{++}(v)$ and we are done. If $v$ is not a $2$-king, then $N^{-}(v)\setminus N^{++}(v)\neq \emptyset$ and $v\in V^-_T$. Since $T$ is a tournament, $N^{++}(v)\subseteq N^-(v)$ and therefore $d^-(v)-d^{++}(v)=|N^{-}(v)\setminus N^{++}(v)|$. Thus, it remains to be shown that  $|N^{-}(v)\setminus N^{++}(v)|=1$. 
	For any $u\in N^{-}(v)\setminus N^{++}(v)$, we have that $u$ dominates $v$ and all vertices in $N^+(v)$. Thus, $d^+(u)\geq d^+(v)+1\geq d$ and so $u\in V^+_T$. If there exists another vertex $u'\in N_T^{-}(v)\setminus N_T^{++}(v)$ then $u'$ dominates $v$ and all vertices in $N^+(v)$ and also $u\in V^+_T$. However,
	either $u$ or $u'$ has out-degree $d$+1 (depending on the direction of the arc between $u$ and $u'$), a contradiction.
	\end{proof}
	
	For any subset $S$ of \AZ{$V(T)$}, \AZ{let $d^+(S)$ denotes the number of arcs from $S$ to $V\setminus S$}.
	Since $d^+(S)-d^-(S)=\sum_{v\in S}(d^+(v)-d^-(v))$, we have the following observation.
	\begin{observation}\label{obs: d+S}
		Let $\AZ{T}=(V,A)$ be an almost regular tournament with order 2d and $S\subseteq V(D)$, then $|d^+(S)-d^-(S)|\leq |S|$. Furthermore, $d^+(S)=d^-(S)+|S|$ ($d^-(S)=d^+(S)+|S|$, respectively) if and only if $S\subseteq V^+_T$  ($S\subseteq V^-_T$, respectively). 
	\end{observation}
	
Now we are ready to prove the first main result of this subsection.

	\begin{thm}\label{thm: ar}
		Let $D=(X, Y,A(D))$ be an oriented split graph, where $Y$ induces an almost regular tournament $T$ with $2d$ vertices. Then $D$ has a Sullivan vertex.
	\end{thm}
	\begin{proof}
		Suppose to the contrary that there is no Sullivan vertex, i.e., for any $u\in V(D), d^{++}(u)<d^{-}(u)$ and in particular $d^-(u)>0$. For any vertex $x\in X$, we may also assume that $d^+(x)>0$ as otherwise, since $x$ is not contained in the in-neighbourhood of any other vertex, the resulting digraph obtained by deleting $x$ still has no Sullivan vertex and then we can consider this digraph instead of $D$.
		
		For any $x\in X$, since $x$ is not a Sullivan vertex, we must have 
		\begin{equation}\label{eq1}
			d^{++}_Y(x)\leq d^{++}_Y(x)+d^{++}_X(x)\leq d^{-}(x)-1,
		\end{equation}
		therefore, if $d^{++}_Y(x)=d^{-}(x)-1$ then $N_X^{++}(x)=\emptyset$.
		
		Let $A_x=N^{+}(x)$, $B_x=N_Y^{++}(x)$ and $C_x=Y-A_x-B_x$. Observe that $C_x\not= \emptyset$ (as otherwise $N^-(x)\subseteq N_Y^{++}(x)$) 
		and $C_x\to A_x$. By Observation \ref{obs: d+S}, we have 
		
		\begin{equation*}
			|A_x||C_x|\leq d_Y^+(C_x)\leq d_Y^-(C_x)+|C_x|\leq |B_x||C_x|+|C_x|,
		\end{equation*}
		which implies
		\begin{equation}\label{eq2}
			d^{++}_Y(x)\geq d^{+}(x)-1.
		\end{equation}
	Equality in (\ref{eq2}) holds if and only if $d_Y^-(C_x)=|B_x||C_x|$ (or equivalently $B_x\to C_x$) and $C_x\subseteq V^+_T$. 
		These two conditions imply that for any vertex $v\in C_x$, $d^+_{D[C_x]}(v)=d-|A_x|$ and therefore $D[C_x]$ is regular. Also recall that $C_x\to N^+(x).$

		Let $C_x'=Y-N^{-}(x)-N_Y^{--}(x)$. If $C_x'=\emptyset$, we have that $N^{+}(x) \subseteq N_Y^{--}(x)$ and thus
		\begin{equation}\label{eq3}
			d^{--}_Y(x)\geq d^{+}(x).
		\end{equation}
		Otherwise, by a similar argument to the one for (\ref{eq2}), we can obtain
		\begin{equation}\label{eq4}
			d^{--}_Y(x)\geq d^{-}(x)-1,
		\end{equation}
		with equality if and only if $C_x' \to N_Y^{--}(x)$ and $C_x'\subseteq V^-_T$. As for (\ref{eq2}), we have that $D[C'_x]$ is regular and  $N^{-}(x) \to C_x'$.
		
		Combining \eqref{eq1}, \eqref{eq2} and \eqref{eq4}, we have
		\begin{equation}\label{eq7}
			d^{--}_Y(x)\geq d^{+}(x)-1,
		\end{equation}
		with equality if and only if equalities in \eqref{eq1}, \eqref{eq2} and \eqref{eq4} hold. Note that if $x$ attains equality in \eqref{eq1} and \eqref{eq2} then 
		\begin{equation}\label{eq5}
			d^+(x)=d^-(x).
		\end{equation}
		Let $X_1=\{x\in X: d^{--}_Y(x)= d^{+}(x)-1\}$. By (\ref{eq3}) and (\ref{eq7}), $X\setminus X_1=\{x\in X: d^{--}_Y(x)\geq d^{+}(x)\}$. If $X_1\not=\emptyset$, we will show the following properties for each $x\in X_1$.
		\begin{description}
			\item[(A0)] $d^+(x)=d^-(x)$.
			
			\item[(A1)]$N_X^{++}(x)=\emptyset$ and therefore $N^{+}(x)\cap N^{-}(y)=\emptyset$ for any $y\in X$.
			\item[(A2)] $D[C_x]$ is regular, $C_x\subseteq V^+_T$ and $N_Y^{++}(x) \to C_x \to N^{+}(x)$.
			\item[(A3)]$D[C_x']$ is regular, $C'_x\subseteq V^-_T$ and $N^{-}(x) \to C_x' \to N_Y^{--}(x)$.
			\item[(A4)]$C_x\subseteq N^{-}(x)$ and $C_x'\subseteq N^{+}(x)$, i.e., $C_x=N^{-}(x)\setminus N^{++}(x)$ and $C'_x= N^{+}(x)\setminus N^{--}(x)$.
			\item[(A5)]For any $x, y\in X_1$, $N^+(x)=N^+(y)$ and $N^-(x)=N^-(y)$. 
		\end{description}
		Properties (A0)--(A3) follow immediately from the equality conditions of  \eqref{eq1}, \eqref{eq2} and \eqref{eq4} and their implications.
		
\vspace{2mm}	
	
		{\bf Proof of (A4): }
		We first prove $C_x\subseteq N^-(x)$ by showing $C_x\cap (N_Y^{--}(x)\cup C_x')=\emptyset$. If there exists $u\in C_x\cap N_Y^{--}(x)$, then by (A2) and (A3), we have that $C_x'\to u \to N^{+}(x)$, which implies $C_x'\cap N^{+}(x)=\emptyset$. Since $C_x'\cap N^{+}(x)=\emptyset$ and $N^{-}(x)\cap N^{+}(x)=\emptyset$, we have $N^{+}(x) \subseteq N_Y^{--}(x)$, which contradicts the fact that $d^{--}_Y(x)= d^{+}(x)-1$. Thus, we have $C_x\cap N_Y^{--}(x)=\emptyset$. 
Properties	(A2), (A3) and the fact that $V^+_T\cap V^-_T=\emptyset$ imply that $C_x\cap C_x'=\emptyset$. Property $C_x'\subseteq N^+(x)$ can be proved analogously. Thus, (A4) is proved.
		
		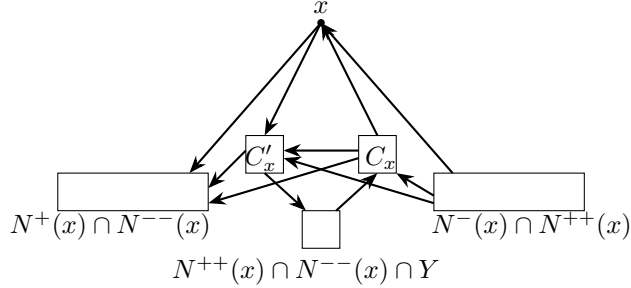
\begin{figure}[htbp]
			\begin{center}
				\begin{tikzpicture}[scale=1]
					\node[fill=black,circle,inner sep=1pt] (t1) at (0,1) {};
					
					\draw (-1,-0.5)--(-0.5,-0.5)--(-0.5,-1)--(-1,-1)--cycle;
					\draw (1,-0.5)--(0.5,-0.5)--(0.5,-1)--(1,-1)--cycle;
					\draw [thick, arrows = {-Stealth[reversed, reversed]}] (0.5,-0.7)--(-0.5,-0.7);
					\draw [thick, arrows = {-Stealth[reversed, reversed]}] (0.5,-0.8)--(-1.5,-1.4);
					\draw [thick, arrows = {-Stealth[reversed, reversed]}] (1.5,-1.3)--(1,-1);
					\draw [thick, arrows = {-Stealth[reversed, reversed]}] (1.5,-1.4)--(-0.5,-0.8);
					\draw [thick, arrows = {-Stealth[reversed, reversed]}] (-1,-0.7)--(-1.5,-1.2);
					
					\draw (-.25,-1.5)--(0.25,-1.5)--(0.25,-2)--(-.25,-2)--cycle;
					\draw (-3.5,-1)--(-1.5,-1)--(-1.5,-1.5)--(-3.5,-1.5)--cycle;
					\draw (3.5,-1)--(1.5,-1)--(1.5,-1.5)--(3.5,-1.5)--cycle;
					
					\draw [thick, arrows = {-Stealth[reversed, reversed]}] (0,1)--(-0.75,-0.5);
					\draw [thick, arrows = {-Stealth[reversed, reversed]}] (0.75,-0.5)--(0,1);
					\draw [thick, arrows = {-Stealth[reversed, reversed]}] (0,1)--(-1.75,-1);
					\draw [thick, arrows = {-Stealth[reversed, reversed]}] (1.75,-1)--(0,1);
					
					\draw [thick, arrows = {-Stealth[reversed, reversed]}] (-.75,-1)--(-.2,-1.5);
					\draw [thick, arrows = {-Stealth[reversed, reversed]}] (.2,-1.5)--(.75,-1);

					\node at (0,1.2) {$x$};
					\node at (-0.8,-0.8) {$C_x'$};
					\node at (0.8,-0.8) {$C_x$};
					\node at (-2.8,-1.7) {$N^{+}(x)\cap N^{--}(x)$};
					\node at (2.8,-1.7) {$N^{-}(x)\cap N^{++}(x)$};
					\node at (-0.2, -2.3) {$N^{++}(x)\cap N^{--}(x)\cap Y$};
				\end{tikzpicture}
				\caption{The graph $D[\{x\}\cup Y]$ with $x\in X_1$.}\label{fig:2}
			\end{center}
		\end{figure}
		By (A0)--(A4), we have that for any $x\in X_1$, $Y$ can be partitioned into five disjoint sets, say $C_x=N^{-}(x)\setminus N^{++}(x), C_x'=N^{+}(x)\setminus N^{--}(x), N^{-}(x)\cap N^{++}(x)$, $N^{+}(x)\cap N^{--}(x)$, $N^{++}(x)\cap N^{--}(x)\cap Y$, see Fig. \ref{fig:2}. 
		
\vspace{2mm}
		
		{\bf Proof of (A5): } We only need to show that for any pair of vertices $x,y\in X_1$  we have that $N^+(y)\subseteq N^+(x)$ ($N^-(y)\subseteq N^-(x)$ can be proved by a similar argument). Suppose that $N^+(y)\not\subseteq N^+(x)$, then there exists an out-neighbour of $y$, $z\in N_Y^{++}(x)\cup C_x$. Note that $N^+(y)\cap C_x=\emptyset$ since if not and $w\in N^{+}(y)\cap C_x$, then by (A4), $y\to w\to x$, a contradiction to (A1). 
		Thus, $z\in N_Y^{++}(x)$. By (A2), $y\to z\to C_x$. In addition, since $N^+(y)\cap C_x=\emptyset$, $C_x\subseteq N_Y^{++}(y)$. Again by (A2), $N_Y^{++}(y)\to C_y$ and therefore $C_x\to C_y$. In particular, $C_x\cap C_y=\emptyset$. But, by (A2), $N_Y^{++}(x) \to C_x \to N^{+}(x)$. Thus, $C_y\subseteq N^+(x)$. By (A4), $x\to C_y\to y$ which means $y\in N_X^{++}(x)$, a contradiction to (A1). This completes the proof of (A5).
		
\vspace{2mm}

		From (A5) we can see that all vertices in $X_1$ have the same neighbourhood. As a result, they have the same set $C_x$. In particular, by (A4), we can see that for any $x\in X_1$, $C_x\to X_1$.
		
		
		
		
		Let $u\in C_x$ and $v\in N_X^{-}(u)\setminus N^{++}(u)$ (the existence of $v$ is guaranteed by the fact that $u$ is not a Sullivan vertex). By (A4), $u\to X_1$ and therefore $X_1\subseteq N^{++}(v)$. In particular, $|X_1|\leq d^{++}(v)$. Since $v$ is not a Sullivan vertex, we have $d^{++}(v)< d^-(v)$. As $v\notin N^{++}(u)$, we have $N^-(v)\cap N^+_Y(u)=\emptyset$ which implies $d^-(v)\leq 2d-d^+_Y(u)\leq d$.
		Combining these inequalities, we have
		\[
		|X_1|<d.
		\]
		
		Recall that $X_1=\{x\in X: d^{--}_Y(x)= d^{+}(x)-1\}$ and $X\setminus X_1=\{x\in X: d^{--}_Y(x)\geq d^{+}(x)\}$.  By Lemma \ref{sums}, $\sum_{x\in X}(d_Y^{--}(x)-d_Y^{+}(x))=\sum_{y\in Y}(d_X^{++}(y)-d_X^{-}(y))$. Thus, we have
		\[
		\sum_{y\in Y}(d_X^{++}(y)-d_X^{-}(y))\geq \sum_{x\in X_1}(d_Y^{--}(x)-d^{+}(x))= -|X_1|.
		\]
		Recall that any vertex in $Y$ is not a Sullivan vertex and $T$ is almost regular, so by Proposition \ref{prop:1}, we have
		\[
		-1\geq d^{++}(y)-d^{-}(y)=\left \{\begin{array}{cc}
			d_X^{++}(y)-d_X^{-}(y)  &  \hbox{if $y$ is a 2-king of $T$;}\\
			d_X^{++}(y)-d_X^{-}(y)-1 & \hbox{otherwise.}
		\end{array}
		\right.
		\]
	Now using the facts that every vertex of $V^+_T$ is a 2-king and $|V^+_T|=d$, 
	we have that 
		\[
		-2d\geq -d+\sum_{y\in Y}(d^{++}_X(y)-d^{-}_X(y))\geq -d-|X_1|, 
		\]
		a contradiction to the fact that $|X_1|<d$. Hence we are done.\end{proof}
	
	Applying Proposition \ref{prop:1}, we can obtain the following similar result.
	\begin{prop}\label{prop:2}
		Let $T$ be an almost regular tournament of order $2d$. For any $v\in V(T)$, we have that 
		\[
		d^{++}(v)-d^{+}(v)=\left \{\begin{array}{cc}
			1  &  \hbox{if $v$ is a 2-king and $d^-(v)=d$,}\\
			-1  &  \hbox{if $d^+(v)=d,$}\\
			0  & \hbox{otherwise.}
		\end{array}
		\right.
		\]
	\end{prop}
	\begin{thm}\label{thm: ars}
	Let $D=(X, Y,A(D))$ be an oriented split graph, where $Y$ induces an almost regular tournament $T$ with $2d$ vertices. Then $D$ has a Seymour vertex.
	\end{thm}
	\begin{proof}
		Suppose to the contrary that there is no Seymour vertex. Thus, for any vertex $x\in X$, $d_X^{++}(x)+d_Y^{++}(x)< d^+(x)$. In particular, $d^+(x)>0$. We may also assume that $d^-(x)>0$ for otherwise $x$ is not contained in the first or second out-neighbourhood of any other vertex, and therefore we can delete it and consider the resulting split digraph. Let $A_x=N^{+}(x)$, $B_x=N_Y^{++}(x)$ and $C_x=Y-A_x-B_x$. One can observe that if $C_x\not= \emptyset$, then $C_x\to A_x$. First we claim that $C_x= \emptyset$. Otherwise, by Observation \ref{obs: d+S}, we have 
		
		\begin{equation*}
			|A_x||C_x|\leq d_Y^+(C_x)\leq d_Y^-(C_x)+|C_x|\leq |B_x||C_x|+|C_x|,
		\end{equation*}
		which implies
		\begin{equation}\label{c0eq1}
			d^{++}_Y(x)\geq d^{+}(x)-1.
		\end{equation}
		We have $d_Y^{++}(x)=d^+(x)-1$ since $x$ is not a Seymour vertex, then all the vertices in $C_x$ with out-degree $d$. Note that $N_X^+(A_x)=\emptyset$, otherwise, $x$ is a Seymour vertex. Choose a vertex $y$ of out-degree $d-1$ in $A_x$ (the number of vertices with out-degree $d-1$ guarantees the existence of such $y$), we have $y$ is a Seymour vertex by the arguments above. 
		
		So in the following we always assume that for any $x\in X$, $C_x=\emptyset$. As a result, for any $x\in X$, since $x$ is not a Seymour vertex and $C_x=\emptyset$, $d^+(x)>d^{++}_Y(x)=2d-d^+(x)$ and therefore $d^+(x)>d$.
		
		Note that $N^-(x)\subseteq N_Y^{++}(x)$, then we have $d^-(x)\leq d_Y^{++}(x)$. Thus,
		\begin{equation}\label{c2eq1}
			d_X^{++}(x)< d^+(x)-d_Y^{++}(x)\leq d^+(x)-d^{-}(x)\leq 2d-2d^-(x).
		\end{equation}
		Now, we consider a vertex in $y\in Y$. Since $y$ is not a Seymour vertex, $d^+_X(y)+d^+_Y(y)=d^+(y)>d^{++}(y)=d^{++}_X(y)+d^{++}_Y(y)$ which implies
		\begin{equation}\label{c2eq2}
			d^+_X(y)-d^{++}_X(y)> d^{++}_Y(y)-d^+_Y(y)\geq -1, 
		\end{equation}
		where the last inequality follows from Proposition \ref{prop:2}.
		
		Let $C_x'=Y-N_D^{-}(x)-N_D^{--}(x)$. We now partition $X$ into two sets $X_a:=\{x\in X: C_{x}'=\emptyset\}$ and $X_b:=X\setminus X_a$. For any $x\in X_a$, since $C_x'=\emptyset$, $N^+(x)\subseteq N^{--}_Y(x)$ which implies 
		\begin{equation}\label{c1eq1}
			d^-(x)< d<d^+(x)\leq d^{--}_Y(x). 
		\end{equation}
		
		By arguments similar to those for  (\ref{c0eq1}), for any $x\in X_b$, we have
		\begin{equation}\label{c2eq3}
			d^{--}_Y(x)\geq d^{-}_Y(x)-1,
		\end{equation}
		with equality if and only if $N^{-}(x) \to C_x' \to N_Y^{--}(x)$, all vertices in $C_x'$ are of in-degree $d$ and $D[C'_x]$ is regular.

		\par Let $X_1:=\{x\in X_b: d^{--}_Y(x)=d^-_Y(x)-1\}$, then $X-X_1=\{x\in X: d^{--}_Y(x)\geq d^-_Y(x)\}$. We want to get a lower and an upper bound for $|X_1|$ in order to achieve a contradiction. We first try to get an upper bound. Now, let $x^*\in X_1$ be a vertex with the minimum in-degree in $X_1$, we partition $X_1$ into two sets $X_{11}:=\{x\in X_1: x\in N^{++}(x^*)\cap X\}$ and $X_{12}:=X_1\setminus X_{11}$. By (\ref{c2eq1}), we have
		\begin{equation}\label{c2eq5}
			|X_{11}|\leq d^{++}_X(x^*)<2d-2d^-(x^*)=|C_{x^*}'|-1.
		\end{equation}
		We now claim that there exists a vertex $u\in N^{--}(x^*)\cap N^+(x^*)$ such that $X_{12}\subseteq N_X^{++}(u)$ (we postpone its proof to the end of the proof of the theorem in order not to break the flow of the proof). Therefore, by (\ref{c2eq2}) we have
		\[|X_{12}|\leq d^{++}_X(u)\leq d^+_X(u)\leq d^{++}_X(x^*),\]
		where the last inequality holds because of $N^+_X(u)\subseteq N^{++}_X(x^*)$. Again applying (\ref{c2eq1}), we can obtain
		\begin{equation}\label{c2eq6}
			|X_{12}|\leq |C_{x^*}'|-1.
		\end{equation}
		By (\ref{c2eq5}) and (\ref{c2eq6}), we have 
		\begin{equation}\label{c2eq7}
			|X_1|=|X_{11}|+|X_{12}|\leq 2|C_{x^*}'|-2.
		\end{equation}

		Now, we try to get the lower bound of $|X_1|$. Let $t$ be the number of vertices in $Y$ which are  2-kings of $T$ with out-degree $d-1$ in $T$. Using Proposition \ref{prop:2}, the assumption that all vertices are not Seymour vertex and Lemma \ref{sums}, we obtain
		\begin{eqnarray*}
			-2d&\geq&\sum_{y \in Y} \left( d^{++}(y)-d^{+}(y) \right)\\ & = & \sum_{y \in Y} \left( d^{++}_Y(y)-d^{+}_Y(y) \right) 
			+ \sum_{y \in Y} \left(d^{++}_X(y)-d^{+}_X(y)\right) \\ \vspace{0.2cm}
			&\AZ{ \geq} &  \AZ{-d+t +  \sum_{x \in X} \left(d^{--}_Y(x)-d^{-}(x) \right)} \\ 
			& \geq &  -d+t-|X_1|,
		\end{eqnarray*}
		i.e., $|X_1|\geq t+d$. On the other hand, we claim that for any $x_1\in X_1$, all vertices in $C_{x_1}'$ are 2-kings of $T$ which implies $t\geq |C_{x_1}'|$ and therefore,
		\begin{equation}\label{c2eq4}
			|X_1|\geq |C_{x_1}'|+d.
		\end{equation}
		In fact, suppose there is a vertex $y'\in C_{x_1}'$ which is not a 2-king of $T$. Note that $y'$ is a 2-king in $D[C_{x_1}']$ since it is regular. Then, because $y'\to N^{--}_Y(x_1)$, there exists a $z\in N^-_Y(x_1)$ which dominates all vertices in $N^{--}_Y(x_1)$. Thus, since $N^-_Y(x_1)\to C_{x_1}'$, $d\geq d^+_Y(z)\geq d^{--}_Y(x_1)+|C_{x_1}'|$ and therefore $d^-(x_1)\geq d$ which contradicts the fact that $d^+(x_1)>d$.
		
		By \eqref{c2eq7} and \eqref{c2eq4}, we have \[
		|C_{x^*}'|+d\leq |X_1|\leq 2|C_{x^*}'|-2,
		\]
		or equivalently $|C_{x^*}'|\geq d+2$. Recall that $N^{-}(x^*)\to C_{x^*}'$, we have that for any $w\in N^{-}(x^*)$, $d^+_Y(w)\geq |C_{x^*}'|\geq d+2$, a contradiction to the fact that $d$ is the maximum out-degree of $T$.

		Now, we end the proof by verifying the following claim.
		\begin{claim}\label{claim1}
			There exists a vertex $u\in N^{--}(x^*)\cap N^+(x^*)$ such that $X_{12}\subseteq N_X^{++}(u)$.
		\end{claim}
		{\bf The proof of Claim \ref{claim1}.} Recall that for all $x\in X$, $d^+(x)>d$. Let $B^*:=N^{--}(x^*)\cap N^+(x^*)$. We first show $B^*\neq \emptyset$. If $B^*= \emptyset$, then $N^+(x^*)\subseteq C_{x^*}'$ which is impossible since $d^+(x^*)>d$ and $|C_{x^*}'|\leq d$. Now we let $u$ be the vertex with the minimum out-degree in $D[B^*]$, then \[|N_Y^+(u)\setminus N^{+}(x^*)|=d^+_Y(u)-d^+_{D[B^*]}(u)\geq d-1- \frac{|B^*|-1}{2}.\] Now we show that $u$ is the required vertex. 
		\par Suppose it is not, then there exists a vertex $x'\in X_{12}\setminus N^{++}(u)$. In particular, $x'\not\in N^{++}(u)\cup N^{++}(x^*)$, i.e., $N^{-}(x')\cap (N^{+}(u)\cup N^{+}(x^*))=\emptyset$.
		Thus we have 
		\[
		d^-(x')<2d-d^+(x^*)-|(N^{+}(u)\cap Y)\setminus N^{+}(x^*)|< \frac{|B^*|+1}{2}.
		\]
		But by the minimality of $d^-(x^*)$,
		$$d^-(x')\geq d^{-}(x^*)= d^{--}_Y(x^*)+1> |B^*|,$$ 
		a contradiction. This completes the proof.
	\end{proof}
	
	\section{Discussion}\label{sec:disc}
	
	We proved Seymour's and Sullivan's conjectures for special classes of oriented graphs. Our results and those of other authors show that both conjectures are very difficult despite their simple formulations. 
	In particular, we were unable to verify either conjecture for all oriented split graphs. As oriented split graphs are a subfamily of multipartite tournaments, which are orientations of complete multipartite graphs, also verifying that the conjectures hold for multipartite tournaments remains an open problem. 
	
	\paragraph{Acknowledgement} This version of the paper corrects a mistake in the proof of Corollary \ref{cor:t} in the previous version of the paper which was pointed to us
	by Emanuele Natale and {\'E}douard Oyallon on July10, 2026.

	\appendix
	
        {\bf \huge Appendix} 
	
	\section{Proof of Proposition \ref{almost}}
	
	For a fixed $u\in V(D)$ and any $v\in V(D)\setminus\{u\}$, the probability of $v\notin N^+(u)\cup N^{++}(u)$ is 
	$(1-\frac{p}{2})(1-\frac{p^2}{4})^{n-2}$. Let $X_u$ be the random variable of the number of vertices that are not in $N^+(u)\cup N^{++}(u)$. Then, 
	\[\mathbb{E}(X_u)=(n-1)(1-\frac{p}{2})(1-\frac{p^2}{4})^{n-2}.\]
	By Markov's inequality,
	\[\mathbb{P}(X_u\geq 1)\leq\mathbb{E}(X_u)=(n-1)(1-\frac{p}{2})(1-\frac{p^2}{4})^{n-2}\to 0~ (as~n\to \infty),\]
	which completes the proof.
	
\end{document}